\documentclass[11pt]{article}

\usepackage{amsfonts,amssymb,amsmath,amsthm}
\usepackage{subfigure, psfrag, dsfont}
\usepackage{hyperref, fullpage, color, upgreek}
\usepackage{graphicx}

\newcommand{\Rbb}{\mathbb{R}}
\newcommand{\scp}[2]{\langle #1, #2 \rangle}

\newtheorem{lemma}{Lemma}

\newcommand{\inv}[1]{\frac{1}{#1}}

\newcommand{\tinv}[1]{{\textstyle\frac{1}{#1}}}

\newcommand{\ud}{\mathrm{d}} 

\newcommand{\Mcal}{\mathcal{M}}

\newcommand{\norm}[1]{\|#1\|}

\DeclareMathOperator{\Id}{Id}
\DeclareMathOperator{\vol}{vol}

\title{Convergence Rate of the\\ Symmetrically Normalized Graph
Laplacian\\[5mm]
{\small TR-LJ-2011.01}}
\author{L. Jacques}
\date{\today}

\begin{document}

\maketitle

\begin{abstract}
  This short note aims at (re)proving that the symmetrically
  normalized graph Laplacian $L=\Id - D^{-1/2}WD^{-1/2}$ (from a graph
  defined from a Gaussian weighting kernel on a sampled smooth manifold)
  converges towards the continuous Manifold Laplacian when the
  sampling become infinitely dense. The convergence rate with
  respect to the number of samples $N$ is $O(1/N)$. 
\end{abstract}

There exist discrete operators which are the equivalent of the
gradient and the divergence operators defined on continuous
manifold. They share with them some common properties and they
converge also to their continuous counterparts for a sufficiently fine
sampling of the underlying manifold.

The first one relies on the definition of edge derivative
\cite{zhou2004rfl, peyre2008nlr}. For a smooth function $f: V\to
\Rbb$, the edge derivative of $f$ on $u\in V$ along the edge
$e=(u,v)\in E$ formed by the connected vertices $u$ and $v\in V$ reads
$$
\nabla_e f\big(u)\ =\ \sqrt{\tfrac{w(u,v)}{2d(v)}}
f(v) - \sqrt{\tfrac{w(u,v)}{2d(u)}} f(u), 
$$
where $d(u) = \sum_{v\in V} w(u,v)$ is the \emph{degree} of the vertex
$u$. From this relation, we have obviously $\nabla_e f\big(u) = -\nabla_e
f\big(v)$. 

The gradient of $f$ is then defined globally as the vector field
$\nabla f: E\to \Rbb$ defined on the edge set $E$ as $\nabla f(e) =
\nabla_{e} f(u)$ for $e=(u,v)\in E$. This gradient can be represented
as a linear operator $\nabla\in\Rbb^{N\times N}$, so that 
$$
\nabla f(u)\ =\ \{\nabla f(u,v):\ v\in V\}\ \in\ \Rbb^N,
$$ 
corresponds to the gradient of $f$ on $u\in V$ seen as a vector of
$\Rbb^N$. The norm of this object on each $u\in V$ is defined
naturally as
$$ 
\norm{\nabla f(u)}^2\ =\ \sum_{v\in V} |\nabla f(u,v)|^2\ =\
\sum_{v\,\sim\,u} |\nabla f(u,v)|^2.  
$$

The scalar products $\scp{f}{g}=\sum_{u\in V} f(u) g(v)$ and
$\scp{F}{G}=\sum_{e\in E} F(e) G(e)$ between two real functions $f$
and $g$ in the Hilbert space $\ell^2(V)=\{h: \sum_{u\in V} |h(u)|^2 <
\infty\}$ and two vector fields $F$ and $G$ in $\ell^2(E)=\{H:
\sum_{e\in E}
|H(e)|^2 < \infty\}$, defined then the adjoint of the gradient,
i.e. the graph divergence $\nabla^*$. Indeed, applying the relation
$\scp{\nabla g}{F}=\scp{g}{-\nabla^* F}$ valid in the continuous
domain \cite{chambolle2008}, we get
$$
[\nabla^* F](u) \ =\ \sum_{v\, \sim\, u} \sqrt{\tfrac{w(u,v)}{2d(u)}}
\big(F(u,v) - F(v,u)\big).
$$ 

The graph Laplacian of a function $f$ is then defined as 
$$
\Delta f(u)\ =\ \nabla^*\nabla f(u)\ =\ \sum_{v\,\sim\, u}
\tfrac{w(u,v)}{\sqrt{d(u)d(v)}} f(v) - f(u). 
$$ 
Using matrix notations, this operator corresponds actually to the
common (symmetric) normalized Laplacian defined on graph, i.e. 
$$
\Delta\ =\ D^{-1/2}WD^{-1/2}\ -\ \Id,
$$
with $D\in\Rbb^{N\times N}$ ($N = \# V$) a diagonal matrix such that
$D_{uu} = d(u)$ and $W\in\Rbb^{N\times N}$ the weight matrix with
$W_{uv} = w(u,v)$.

Interestingly, the graph Laplacian converges to the continuous
Laplace-Beltrami operator \cite{carmo1992rg} on the manifold
underlying the graph definition.

\begin{lemma}
  If the vertex set $V=\{v_1,\,\cdots,v_N\}\subset\,\Mcal$
  consists of $N$ points taken uniformly and independently at random
  on a $m$-dimensional compact manifold $\Mcal$ embedded in
  $\Rbb^n$, then, for any smooth function $f:V\to \Rbb$ and for a
  weighting function $w(u,v)=\exp\{ -\norm{u-v}^2/{2\epsilon}\}$, the
  normalized graph Laplacian defined by the graph $G=(V,E=V\times
  V,w)$ satisfies
$$
\tinv{\epsilon}\,\lim_{N\to\infty}\Delta f(u)\ =\
\tinv{2}\,\Delta_{\Mcal} f(u)\ +\ O(\epsilon^{1/2}), 
$$
with $\Delta_{\Mcal}$ the Laplace-Beltrami operator defined on
$\Mcal$.
\label{lem:discr_lapl_manifold_lapl}
\end{lemma}

\begin{proof}
  We follow a similar development to the one given in \cite{singer2006gml}. With
  the hypothesis of the Lemma, for any function $h\in \ell^2(V)$, we
  have
\begin{align*}
  [Wh](u)\ =\ \sum_j w(u,v_j)\,h(v_j)&=\sum_{j\neq i} \exp\{
  -\norm{u-v_j}^2/{2\epsilon}\}\,h(v_j)\ +\ h(u).
\end{align*}
Since $Y_j=w(u,v_j)\,h(v_j)$ are iid for $j\neq i$, by the
law of large numbers we have
\begin{align*}
  \sum_{j\neq i} w(u,v_j)\,h(v_j)&\simeq\ (N-1)\,
  \mathbb{E}_\Mcal\big[\exp\{u-\cdot\}h(\cdot)\big]\\
&=\ \tfrac{N-1}{\vol\Mcal}\,\int_\Mcal \exp\{
  -\norm{u-y}^2/{2\epsilon}\}\,h(y)\ \ud_\Mcal y,
\end{align*}
where the integral is performed on the manifold with the local
infinitesimal volume element $\ud_\Mcal y$. The relation (2.9) in
\cite{singer2006gml} (or Eq. (10) in \cite{smolyanov2000bmm}) explains
that
$$
\tfrac{1}{(2\pi\epsilon)^{m/2}}\,\int_\Mcal \exp\{
  -\norm{u-y}^2/{2\epsilon}\}\,h(y)\ \ud_\Mcal y\ =\ h(u) +
  \tfrac{\epsilon}{2} [E(u)h(u) + \Delta_\Mcal h(u)] + O(\epsilon^{3/2}),
$$    
where $E(u)=\inv{3}S(u)$ and $S$ is the scalar curvature of $\Mcal$
\cite{carmo1992rg}. Therefore,
\begin{align}
  [Wh](u)&=\ \tfrac{(N-1)(2\pi\epsilon)^{m/2}}{\vol\Mcal}\big[h(u)
  + \tfrac{\epsilon}{2} [E(u)h(u) + \Delta_\Mcal h(u)] +
  O(\epsilon^{3/2}) \big] + h(u)\nonumber\\
 &=\ \tfrac{(N-1)(2\pi\epsilon)^{m/2}}{\vol\Mcal}\big[h(u)
  + \tfrac{\epsilon}{2} [E(u)h(u) + \Delta_\Mcal h(u)] +
  O\big(\epsilon^{3/2}, 1/(N\epsilon^{m/2})\big)\big],
\label{eq:Wh-tayl-dev}
\end{align}
where the notation $O(\mu,\nu)$ means $|O(\mu,\nu)| < C\mu + D\nu$ for
two positive values $\mu,\nu \ll 1$. Taking $h=1$, we get then
\begin{equation}
\label{eq:degree-dev}
d(v_j)  = \tfrac{(N-1)(2\pi\epsilon)^{m/2}}{\vol\Mcal}\big[1
  + \tfrac{\epsilon}{2}\,E(v_j) + O\big(\epsilon^{3/2}, 1/(N\epsilon^{m/2})\big)\big].
\end{equation}
Therefore, 
\begin{align*}
\Delta f(u)&=\ \sum_j \tfrac{w(u,v_j)}{\sqrt{d(u) d(v_j)}} f(v_j)\ -\
f(u)\\
&=\ \tfrac{1}{\sqrt{1 +
\frac{\epsilon}{2}E(u)}}\ \tfrac{\vol\Mcal}{(N-1)(2\pi\epsilon)^{m/2}} \sum_{j\neq i}
w(u,v_j)\,\tfrac{f(v_j)}{\sqrt{1 + \frac{\epsilon}{2}E(v_j)}}\ -\
f(u)\ +\ O\big(\epsilon^{3/2}, 1/(N\epsilon^{m/2})\big). 
\end{align*}
Taking now $h(t)=f(t)g(t)$ for $t\in\Mcal$ with
$g(t)=\tfrac{1}{\sqrt{1 + \frac{\epsilon}{2}E(t)}}$ in
\eqref{eq:Wh-tayl-dev}, we get
\begin{multline*}
\tfrac{\vol\Mcal}{(N-1)(2\pi\epsilon)^{m/2}}\ \sum_{j\neq i}
w(u,v_j)\,f(v_j)g(v_j)\\ 
=\ \big[ g(u)f(u)(1
  + \tfrac{\epsilon}{2} E(u)) + \tfrac{\epsilon}{2}\Delta_\Mcal h(u) +
  O\big(\epsilon^{3/2}, 1/(N\epsilon^{m/2})\big)\big].
\end{multline*}
However, 
$$
\Delta_\Mcal h(u) = f(u)\Delta_\Mcal g(u) + 2\,\scp{\nabla_\Mcal f(u)}{\nabla_\Mcal g(u)}_{T_u\Mcal} +
g(u) \Delta_\Mcal f(u),
$$
where the scalar product $\scp{\cdot}{\cdot}_{T_u\Mcal}$ occurs in the
tangent plane $T_u\Mcal\simeq \Rbb^d$ of $\Mcal$ on $u$ with the
metric of the manifold \cite{carmo1992rg}. For any function $s$ on
$\Mcal$, the gradient $\nabla_\Mcal s(u) \in T_u\Mcal$ is composed of
the derivative of $s$ according to a local system of coordinates (or
\emph{chart}) isomorphic to $\Rbb^d$.

From the definition of $g$, it is clear that any differential operator
$D_\Mcal$ of $g$ with respect to this local system in $T_u\Mcal$ is of
order $D_\Mcal g(u) = O(\epsilon)$. Consequently,
$$
\tfrac{\epsilon}{2}\Delta_\Mcal h(u) = \tfrac{\epsilon}{2}g(u) \Delta_\Mcal f(u) + O(\epsilon^2),
$$
that provides the final result,   
\begin{align*}
\Delta f(u)&=\ g^2(u) f(u)\,\big(1 +
\tfrac{\epsilon}{2}E(u)\big)\ +\
\tfrac{\epsilon}{2}\Delta_\Mcal f(u)\ -\ f(u)\ +\ O\big(\epsilon^{3/2}, 1/(N\epsilon^{m/2})\big)\\
&=\ \tfrac{\epsilon}{2}\,\Delta_\Mcal f(u)\ +\ O\big(\epsilon^{3/2}, 1/(N\epsilon^{m/2})\big). 
\end{align*}

\end{proof}

% Generated by IEEEtran.bst, version: 1.13 (2008/09/30)

\end{document}